\documentclass[11pt]{article}
\usepackage{amsmath,amsfonts,amssymb,amsthm}
\begin{document}
\author{\textbf{Denis~I.~Saveliev}
}
\title{\textbf{Ultrafilter extensions of linear orders}}
\date{July 2013}
\maketitle

\theoremstyle{plain}
\newtheorem{thm}{Theorem}
\newtheorem*{tm}{Theorem}
\newtheorem{lmm}{Lemma}
\newtheorem*{lm}{Lemma}
\newtheorem*{prp}{Proposition}
\newtheorem{coro}{Corollary}
\newtheorem*{cor}{Corollary}

\theoremstyle{definition}
\newtheorem*{df}{Definition}
\newtheorem*{exm}{Example}
\newtheorem*{prb}{Problem}
\newtheorem*{ack}{Acknowledgements}

\theoremstyle{remark}
\newtheorem*{pf}{Proof}
\newtheorem*{rmk}{Remark}
\newtheorem*{q}{Question}
\newtheorem*{qs}{Questions}
\newtheorem*{ex}{Example}
\newtheorem*{exs}{Examples}
\newtheorem*{fact}{Fact}
\newtheorem*{facts}{Facts}
\newtheorem*{task}{Task}

\newcommand{\Fr}{\mathop{\mathrm {Fr}}\nolimits}
\newcommand{\ind}{\mathop{\mathrm {ind}}\nolimits}
\newcommand{\Ind}{\mathop{\mathrm {Ind}}\nolimits}
\newcommand{\spec}{\mathop{\mathrm {spec}}\nolimits}
\newcommand{\specm}{\mathop{\mathrm {specm}}\nolimits}
\newcommand{\fix}{\mathop{\mathrm {fix}}\nolimits}
\newcommand{\pcf}{\mathop{\mathrm {pcf}}\nolimits}
\newcommand{\crit}{\mathop{\mathrm {crit}}\nolimits}
\newcommand{\supp}{\mathop{\mathrm {supp}}\nolimits}
\newcommand{\cs}{ {\mathop{\mathrm {cs}}\nolimits} }
\newcommand{\club}{ {\mathop{\mathrm {club}}\nolimits} }
\newcommand{\id}{ {\mathop{\mathrm {id}}\nolimits} }
\newcommand{\pr}{ {\mathop{\mathrm {pr}}\nolimits} }
\newcommand{\cnc}{ {^\frown} }

\newcommand{\cl}{ {\mathop{\mathrm {cl\,}}\nolimits} }
\newcommand{\cof}{ {\mathop{\mathrm {cof\,}}\nolimits} }
\newcommand{\add}{ {\mathop{\mathrm {add\,}}\nolimits} }
\newcommand{\sat}{ {\mathop{\mathrm {sat\,}}\nolimits} }
\newcommand{\tc}{ {\mathop{\mathrm {tc\,}}\nolimits} }
\newcommand{\unif}{ {\mathop{\mathrm {unif\,}}\nolimits} }
\newcommand{\fp}{ {\mathop{\mathrm {fp\,}}\nolimits} }
\newcommand{\fs}{ {\mathop{\mathrm {fs\,}}\nolimits} }
\newcommand{\dom}{ {\mathop{\mathrm {dom\,}}\nolimits} }
\newcommand{\ran}{ {\mathop{\mathrm{ran\,}}\nolimits} }
\newcommand{\cf}{ {\mathop{\mathrm {cf\,}}\nolimits} }
\newcommand{\ot}{ {\mathop{\mathrm {ot\,}}\nolimits} }
\newcommand{\image}{\/``\,}
\newcommand{\scc}{\beta\!\!\!\!\beta}
\newcommand{\FS}{ {\mathop{\mathrm {FS\,}}\nolimits} }
\newcommand{\FP}{ {\mathop{\mathrm {FP\,}}\nolimits} }

\newcommand{\uhr}{\!\upharpoonright\!}
\newcommand{\lra}{ {\leftrightarrow} }
\newcommand{\ol}{\overline}

\newcommand{\bigforall}{ {\mathop{\raisebox{-2pt}{\text{\Large$\forall$\/}}}} }
\newcommand{\bigexists}{ {\mathop{\raisebox{-2pt}{\text{\Large$\exists$\/}}}} }
\newcommand{\btd}{ {\mathop{\raisebox{1pt}{\text{\Large$\bigtriangledown$\/}\!}}} }
\newcommand{\btu}{ {\mathop{\raisebox{-2pt}{\text{\Large$\bigtriangleup$\/}}}} }
\newcommand{\bbtd}{ {\mathop{\raisebox{1pt}{\text{\LARGE$\bigtriangledown$\/}\!}}} }
\newcommand{\bbtu}{ {\mathop{\raisebox{-2pt}{\text{\LARGE$\bigtriangleup$\/}}}} }

\newcommand{\Lim}{\mathrm{Lim}}
\newcommand{\Ord}{ {\mathrm {Ord}} }
\newcommand{\Card}{ {\mathrm {Card}} }
\newcommand{\Reg}{\mathrm{Reg}}
\newcommand{\C}{\mathrm{C}}
\newcommand{\CA}{ {\mathrm {CA}} }
\newcommand{\J}{ {\mathrm J} }
\newcommand{\PC}{ {\mathrm {PC}} }
\newcommand{\PRA}{ {\mathrm {PRA}} }
\newcommand{\T}{ {\mathrm T} }
\newcommand{\TA}{ {\mathrm {TA}} }
\newcommand{\PA}{ {\mathrm {PA}} }
\newcommand{\KP}{ {\mathrm {KP}} }
\newcommand{\Z}{ {\mathrm Z} }
\newcommand{\ZF}{ {\mathrm {ZF}} }
\newcommand{\ZFA}{ {\mathrm {ZFA}} }
\newcommand{\ZFC}{ {\mathrm {ZFC}} }
\newcommand{\AEx}{ {\mathrm {AE}} }
\newcommand{\AR}{ {\mathrm {AR}} }
\newcommand{\WR}{ {\mathrm {WR}} }
\newcommand{\AF}{ {\mathrm {AF}} }
\newcommand{\AC}{ {\mathrm {AC}} }
\newcommand{\GC}{ {\mathrm {GC}} }
\newcommand{\DC}{ {\mathrm {DC}} }
\newcommand{\AD}{ {\mathrm {AD}} }
\newcommand{\AInf}{ {\mathrm {AInf}} }
\newcommand{\AU}{ {\mathrm {AU}} }
\newcommand{\AP}{ {\mathrm {AP}} }
\newcommand{\PI}{ {\mathrm {PI}} }
\newcommand{\CH}{ {\mathrm {CH}} }
\newcommand{\GCH}{ {\mathrm {GCH}} }
\newcommand{\APr}{ {\mathrm {APr}} }
\newcommand{\ASp}{ {\mathrm {ASp}} }
\newcommand{\ARp}{ {\mathrm {ARp}} }
\newcommand{\AFA}{ {\mathrm {AFA}} }
\newcommand{\RK}{ {\mathrm {RK}} }

\begin{abstract}
Ultrafilter extensions of arbitrary first-order models were defined 
in~\cite{Saveliev}. Here we consider the case when the models are linearly 
ordered sets. We explicitly calculate the extensions of a~given linear order 
and the corresponding operations of minimum and maximum on a~set. We show 
that the extended relation is not more an order but is close to the natural 
linear ordering of nonempty half-cuts of the set and that the two extended 
operations define a~skew lattice structure on the set of ultrafilters. 
\end{abstract}


\subsection*{1. Preliminaries}

Ultrafilter extensions of arbitrary first-order models were defined 
in~\cite{Saveliev}. If $(X,F,\ldots,P,\ldots)$ is a~model with the 
universe~$X$, operations~$F,\ldots$\,, and relations~$P,\ldots$\,, 
it canonically extends to the model $(\scc X,\widetilde F,\ldots,
\widetilde P,\ldots)$ (of the same language), where $\scc X$~is 
the set of ultrafilters over~$X$, the operations $\widetilde F,\ldots$
extend the operations $F,\ldots$\,, and the relations $\widetilde P,\ldots$
extend the relations $P,\ldots$\,. Here $X$~is considered as a~subset 
of~$\scc X$ by identifying each element~$x$ in~$X$ with the principal 
ultrafilter~$\widetilde x$ given by~$x$. The main result of~\cite{Saveliev} 
shows that, roughly speaking, the construction smoothly generalizes the 
Stone--\v{C}ech compactification of a~discrete space to the situation 
when the space carries a~first-order structure.

The principal precursor of this construction was 
ultrafiter extensions of semigroups, the technique invented in~60s 
and then used to obtain significant results in number theory, algebra, 
and topological dynamics; the book~\cite{Hindman Strauss} is 
a~comprehensive treatise of this field. For the general definition of 
the extension, a~description of topological properties of the extended 
models, and the precise formulation of the aforementioned result,
we refer the reader to~\cite{Saveliev}.

In this note we consider a~rather special case of models, namely, 
linearly ordered sets. We shall deal only with binary relations and 
operations. If $R$~is a~binary relation on a~set~$X$, it extends 
to the binary relation~$\widetilde R$ on the set~$\scc X$ defined by 
\begin{align*}
u\,\widetilde R\:v
\;\;\lra\;\;
\bigl\{x\in X:\{y\in X:x\,R\:y\}\in v\bigr\}\in u
\end{align*}
for all ultrafilters $u,v\in\scc X$, and if $F$~is a~binary operation on~$X$, 
it extends to the binary operation~$\widetilde F$ on~$\scc X$ defined by  
\begin{align*}
S\in\widetilde F(u,v)
\;\;\lra\;\;
\bigl\{x\in X:\{y\in X:F(x,y)\in S\}\in v\bigr\}\in u
\end{align*}
for all $u,v\in\scc X$ and all $S\subseteq X$. 
The relations and operations considered here are definable from a~given 
linear order~$<$, namely, the orders $<$ and~$\le$, the converse orders 
$>$ and~$\ge$, and the operations of minimum and maximum.


As~usually, a~transitive binary relation is a~{\it pre-order\/} 
({\it order\/}, {\it strict order\/}) iff it is reflexive (reflexive 
and antisymmetric, irreflexive); it is called {\it linear\/}, or {\it total\/}, 
iff it is connected, i.e.~any two distinct elements are comparable. 
We use the standard notation $\le$~for (pre-)orders and $<$~for 
strict orders, and its variants. By a~{\it linearly ordered set\/} we mean 
as $(X,\le)$ as well as $(X,<)$, and usually write simply~$X$. 
A~subset~$I$ of a~linearly ordered~$X$ is its {\it initial 
segment\/} iff it is downward closed, i.e.~iff $y\in I$ implies $x\in I$ 
for all $x<y$; {\it final segments\/} are upward closed subsets.
A~pair $(I,J)$ is a~{\it cut\/} of a~linearly ordered set~$X$ iff
$I$ and~$J$ are an initial and a~final segments of~$X$ forming 
its partition (so $x<y$ for all $x\in I$ and $y\in J$); we shall call 
$I$ and~$J$ the (left and right) {\it half-cuts\/}.
A~cut $(I,J)$ is {\it proper\/} iff both $I$ and~$J$ are nonempty,
a~{\it jump\/} iff $I$~has the greatest element and $J$~the smallest one,
a~{\it gap\/} iff neither of these two happens, and a~{\it Dedekind cut\/} 
iff only one happens, i.e.~either $I$~has the greatest element but $J$~does 
not have the smallest one, or conversely. A~linearly ordered set is 
{\it dense\/} iff it has no jumps, {\it complete\/} iff it has no proper gaps, 
and {\it continuous\/} iff it is dense and complete (so has only Dedekind 
cuts). The {\it Dedekind completion\/} of~$X$ is the smallest complete set 
containing~$X$, which is obtained by adding to~$X$ all its proper gaps; 
if one adds also improper gaps, the resulting set is the smallest ordered 
compactification of~$X$ (w.r.t.~the interval topology). Arbitrary ordered 
compactification of~$X$ has either one or two elements filling each proper 
gap of~$X$, so the family of all ordered compactifications of~$X$ is isomorphic 
to the powerset of the set of its proper gaps (see \cite{Fedorchuk66}--\cite{Kaufman} 
and recent review~\cite{Bezhanishvili Morandi}). For more on linearly ordered sets 
we refer the reader to~\cite{Rosenstein}.


N.~L.~Poliakov asked me about the ultrafilter extensions of 
a~linear order and the corresponding operations of minimum and maximum. 
He hypothesed\,\footnote{Personal communication. July, 2013.} that 
the ultrafilter extension of a~linear order on a~set is a~linear pre-order 
whose quotient is isomorphic to the natural ordering of cuts of the set. 
Here it is proved that his attempt to describe the extension works 
for well-orders and in general is, though not correct, rather close 
to be correct: the extension itself is not a~pre-order but a~certain its 
combination with the extension of the converse order gives a~pre-order 
whose quotient is isomorphic to the natural ordering of half-cuts.

The structure of this note is as follows.
In Section~2 we define, for every ultrafilter~$u$ over a~linearly ordered set,
its support, which is either the element generating~$u$ if $u$~is principal, 
or a~half-cut otherwise. We note that the natural linear ordering of supports 
connects to the largest linearly ordered compactification of the set.
In Section~3 we describe the ultrafilter extensions of given order relations 
$<,\le,>,\ge$ in terms of supports of ultrafilters. Then we show that 
$\widetilde<,\widetilde\le,\widetilde>,\widetilde\ge$ do not share many 
features of orders, however, can be ``amalgamed" into a~linear pre-order 
inducing the natural ordering of supports. In Section~4 we describe ultrafilter 
extensions of the operations $\min$ and $\max$ in terms of supports, show that, 
except for commutativity, $\widetilde\min$ and $\widetilde\max$ have the usual 
features of $\min$ and $\max$ on a~linearly ordered set and, actually, turn out 
the set of ultrafilters into a~distributive skew lattice of a~special form. 
Finally, we show that the equivalence~$D$ on the skew lattice coincides with 
the equality of supports and the quotient lattice $\scc X/D$ is isomorphic
to the set of supports with its operations $\min$ and~$\max$. We conclude by 
asking about properties of ultrafilter extensions of partially ordered sets 
and related algebras. The note is quite easy and self-contained.


\subsection*{2. Supports of ultrafilters over linearly ordered sets}

Let $X$~be a~linearly ordered set.
For any ultrafilter~$u$ over~$X$ define
the initial segment $I_u$ and the final segment $J_u$
of~$X$ as follows:
\begin{align*}
I_u
&=\,
\bigcap\,\{
I\in u:
I\text{ is an initial segment of }X\},
\\
J_u
&=\,
\bigcap\,\{
J\in u:
J\text{ is a~final segment of }X\}.
\end{align*}

\begin{lmm}
Let $X$~be a~linearly ordered set 
and\, $u$~an ultrafilter over~$X$.
\\
1. 
If $u$~is principal, then $I_u\cap J_u=\{x\}$ where $u=\widetilde x$. 
\\
2. 
If $u$~is non-principal, then $(I_u,J_u)$ is a~cut,
and either $I_u$ or $J_u$, but not both, is in~$u$.
\\
3.
If $I_u$ is in~$u$, then so are all final segments of~$I_u$,\,
$S\cap I_u$ is cofinal in~$I_u$ for all $S\in u$, and 
$I_u$~does not have the greatest element whenever $u$~is non-principal.
\\ 
4. 
If $J_u$ is in~$u$, then so are all initial segments of~$J_u$,\, 
$S\cap J_u$ is coinitial in~$J_u$ for all $S\in u$, and 
$J_u$~does not have the least element whenever $u$~is non-principal. 
\end{lmm}

\begin{proof} 
Easy.
\end{proof}


Define the {\it support\/} $\supp(u)$ of an ultrafilter~$u\in\scc X$ by
\begin{align*}
\supp(u)=
\left\{
\begin{array}{cl}
\{x\}&\text{if }\,u=\widetilde x
\text{ and }x\in X,  
\\
\rule{0em}{1.2em}
I_u&\text{if }\,u\in\scc X\setminus X
\text{ and }I_u\in u,  
\\
\rule{0em}{1.2em}
J_u
&\text{if }\,u\in\scc X\setminus X
\text{ and }J_u\in u.
\end{array}
\right.
\end{align*}
Thus supports of ultrafilters over~$X$ are subsets of~$X$ which are
either singletons, or initial segments without the last point, or else 
final segments without the first point, and it is clear that any subset 
of one of the three forms is the support of some ultrafilter.

\begin{ex}
If $X$~is well-ordered and $u\in\scc X\setminus X$, then $\supp(u)=I_u$.
If $\alpha$~is an ordinal and $u\in\scc\alpha\setminus\alpha$, then 
$\supp(u)$~is a~limit ordinal~$\beta\le\alpha$.
\end{ex}

This notion of supports, however, should be slightly refined. 
Let $X$~have no end-points (e.g.~$X$~is the set~$\mathbb Z$ of integers 
with their natural ordering), and let $u\in\scc X$~have all initial segments 
of~$X$ and $v\in\scc X$~all final segments of~$X$. Then $\supp(u)=J_u=X$ and 
$\supp(v)=I_v=X$, which shows that our notion cannot distinguish ultrafilters 
``concentrated" at the beginning and at the end of the set.
There are several ways to correct this. E.g.~in such cases we could define
the supports as $\{-\infty\}$ and $\{+\infty\}$ (in fact, adding end-points 
to the set); or we could define the support of an~$u$ as a~pair~--- either 
$(I_u,J_u)$ or $(J_u,I_u)$ depending on what of $I_u$ and~$J_u$ is in~$u$.
We prefer, however, to keep the definition above but understand henceforth 
the expressions ``$\supp(u)=I_u$" by ``$u$~is non-principal and all final 
segments of~$I_u$ are in~$u$" and  ``$\supp(u)=J_u$" by ``$u$~is non-principal 
and all initial segments of~$J_u$ are in~$u$".


The set of supports carries a~natural linear order:
$\supp(u)<\supp(v)$ iff either the cut given by $\supp(u)$ is less than 
the cut given by $\supp(v)$, or $\supp(u)$~is the initial segment and 
$\supp(v)$~is the final segment of the same cut. 
All possible cases are listed in the following table:
\begin{align}\label{ordering supports}
\begin{array}{r|rrr}
\supp(u)<\supp(v)
&\;\supp(v)=\{y\}&\;\supp(v)=I_v\;\;&\supp(v)=J_v\;
\phantom{\underbrace{|}}
\\
\hline    
\supp(u)=\{x\}\;&
\;x<y\quad&\;x<\sup I_v&\;\;x\le\inf J_v
\phantom{\overbrace{|}}    
\\
\supp(u)=I_u\;\;\;\,&
\;\sup I_u\le y\quad&\;\sup I_u<\sup I_v&\;\;\sup I_u\le\inf J_v
\phantom{\overbrace{0}}    
\\
\supp(u)=J_u\;\;\;&
\;\inf J_u<y\quad&\;\inf J_u<\sup I_v&\;\;\inf J_u<\inf J_v
\phantom{\overbrace{0}}    
\end{array}
\end{align}
(which should be read as follows: 
``if $\supp(u)=\{x\}$ and $\supp(v)=\{y\}$, 
then $\supp(u)<\supp(v)$ is equivalent to $x<y$", etc.)~providing 
that $\sup$ and $\inf$ are in the Dedekind completion of~$X$.

\par 

Given a~linearly ordered set~$X$, let $s(X)$~denote the set of 
the supports of ultrafilters over~$X$ with their natural ordering. 
The transition from a~linearly ordered set~$X$ to the linearly ordered 
set~$s(X)$ is a~procedure similar to the Dedekind completion of~$X$ or, 
rather, the ordered compacti\-fi\-ca\-tion of~$X$; however, while the 
latter two add to the set only its gaps, the former one adds all its 
unbounded {\it half-cuts\/} (rather than cuts), i.e.~initial segments 
without the greatest element and final segments without the least element.
Note also that both completion and compactification procedures are idempotent 
(i.e.~their iterations do not change sets) while our construction is not.

If $X$~is of the order-type~$\tau$, let $s(\tau)$ denote the order-type 
of~$s(X)$. As it is customarily in linear order theory, the letters $\zeta$, 
$\eta$, $\lambda$ are used to denote the order-types of the sets $\mathbb Z$, 
$\mathbb Q$, $\mathbb R$ of integers, rationals, reals, respectively;
the multiplication of order-types is antilexicographic (e.g.~$2\omega=\omega$,
$\omega2=\omega+\omega$); for more details see~\cite{Rosenstein}.

\begin{exs}

1. 
$s(\omega)=\omega+1$. Moreover, for all ordinals~$\alpha$,
$s(\omega+\alpha)=\omega+\alpha+1$.

2. 
$s(\zeta)=1+\zeta+1$. Moreover, for all ordinals~$\alpha$,
$s(\alpha\zeta)=1+\alpha\zeta+1$.

3. 
$s(\lambda)=1+3\lambda+1$. Moreover, for all continuous 
order-types~$\tau$, $s(\tau)=1+3\tau+1$.

4. 
$s(\eta)=1+\sum_{x\in\mathbb R}\tau_x+1$ where 
$\tau_x=3$ if $x\in\mathbb Q$, and $\tau_x=2$ otherwise.
Moreover, for all dense order-type~$\tau$,
$s(\tau)=1+\sum_{x\in Y}\tau_x+1$ where 
$\tau_x=3$ if $x\in X$, and $\tau_x=2$ otherwise,
whenever $X$~is any set of the order-type~$\tau$ 
and $Y$~the Dedekind completion of~$X$. 
\end{exs}


\subsection*{3. Ultrafilter extensions of linear orders}

The following theorem describes the ultrafilter extensions
of linear orders in terms of supports.

\begin{thm}
For all ultrafilters~$u,v$ over a~linearly ordered set~$X$,
\begin{align*}
u\:{\widetilde<}\:v
&\;\;\lra\;\;
\supp(u)<\supp(v)\;\vee\;
\supp(u)=\supp(v)=I_u=I_v,
\\
u\:{\widetilde\le}\:v
&\;\;\lra\;\;
\supp(u)<\supp(v)\;\vee\;
\supp(u)=\supp(v)=I_u=I_v\;\vee\;
\exists x\,(u=v=\widetilde x),
\\
\rule{0em}{1em}
u\:{\widetilde>}\:v
&\;\;\lra\;\;
\supp(u)>\supp(v)\;\vee\;
\supp(u)=\supp(v)=J_u=J_v,
\\
u\:{\widetilde\ge}\:v
&\;\;\lra\;\;
\supp(u)>\supp(v)\;\vee\;
\supp(u)=\supp(v)=J_u=J_v\;\vee\;
\exists x\,(u=v=\widetilde x).
\end{align*}
Consequently, on non-principal ultrafilters, 
${\widetilde<}$~coincides with~${\widetilde\le}$ and 
${\widetilde>}$~coincides with~${\widetilde\ge}$.
\end{thm}

\begin{proof} 
Let $X_{<x}$~denote the initial segment $\{y\in X:y<x\}$, 
and $X_{\le x}$, $X_{>x}$, $X_{\ge x}$ have the expected meaning.
By definition, $u\,{\widetilde<}\,v$ means $\{x:X_{>x}\in v\}\in u$. 
First, we observe that 
\begin{align*}
\qquad\quad\;\;\;
\begin{array}{c|ccc}
&\;\supp(v)=\{y\}&\;\supp(v)=I_v\;\;&\;\supp(v)=J_v\;
\phantom{\underbrace{|}}
\\
\hline    
X_{<x}\in v\;&
y<x&\sup I_v\le x&\;\inf J_v<x
\phantom{\overbrace{0}}    
\\
X_{\le x}\in v\;&
y\le x&\sup I_v\le x&\;\inf J_v<x
\phantom{\overbrace{0}}    
\\
X_{>x}\in v\;&
x<y&x<\sup I_v&\;x\le\inf J_v
\phantom{\overbrace{0}}    
\\
X_{\ge x}\in v\;&
x\le y&x<\sup I_v&\;x\le\inf J_v
\phantom{\overbrace{0}}    
\end{array}
\end{align*}
(should be read: ``if $\supp(v)=\{y\}$, 
then $X_{<x}\in v$ is equivalent to $y<x$", etc.),
and so
\begin{align}\label{table 2}
\begin{array}{c|ccc}
&\;\supp(v)=\{y\}&\supp(v)=I_v&\supp(v)=J_v
\phantom{\underbrace{|}}
\\
\hline    
\{x:X_{<x}\in v\}\;&
X_{>y}&X_{\ge\sup I_v}&X_{>\inf J_v}
\phantom{\overbrace{0}}    
\\
\{x:X_{\le x}\in v\}\;&
X_{\ge y}&X_{\ge\sup I_v}&X_{>\inf J_v}
\phantom{\overbrace{0}}    
\\
\{x:X_{>x}\in v\}\;&
X_{<y}&X_{<\sup I_v}&X_{\le\inf J_v}
\phantom{\overbrace{0}}    
\\
\{x:X_{\ge x}\in v\}\;&
X_{\le y}&X_{<\sup I_v}&X_{\le\inf J_v}
\phantom{\overbrace{0}}    
\end{array}
\end{align}
(should be read: ``if $\supp(v)=\{y\}$, 
then $\{x:X_{<x}\in v\}$ equals $X_{>y}$", etc.).
Repeating this observation once more, we characterize
$\{x:X_{>x}\in v\}\in u$ as follows:
\begin{align*}
\quad
\begin{array}{l|rrr}
\{x:X_{>x}\in v\}\in u&
\;\supp(v)=\{y\}&\supp(v)=I_v\;&\;\supp(v)=J_v\;
\phantom{\underbrace{|}}
\\
\hline    
\quad\supp(u)=\{x\}&
x<y\quad&x<\sup I_v&\quad x\le\inf J_v
\phantom{\overbrace{0}}    
\\
\quad\supp(u)=I_u&
\sup I_u\le y\quad&\sup I_u\le\sup I_v&\quad\sup I_u\le\inf J_v
\phantom{\overbrace{0}}    
\\
\quad\supp(u)=J_u&
\inf J_u<y\quad&\inf J_u<\sup I_v&\quad\inf J_u<\inf J_v
\phantom{\overbrace{0}}    
\end{array}
\end{align*}
(should be read: ``if $\supp(u)=\{x\}$ and $\supp(v)=\{y\}$, 
then $\{x:X_{>x}\in v\}\in u$ is equivalent to $x<y$", etc.).
And comparing this with~(\ref{ordering supports}), we see that 
$\{x:X_{>x}\in v\}\in u$ holds iff either $\supp(u)<\supp(v)$ 
or $\supp(u)=I_u=\supp(v)=I_v$, as required.

\rule{0em}{0.4em}

Next, we have 
\begin{align*}
u\,{\widetilde\le}\,v
\;\;\lra\;\;
\bigl\{x:\{y:x\le y\}\in v\bigr\}\in u
\;\;\lra\;\;
\bigl\{x:\{y:x<y\,\vee\, x=y\}\in v\bigr\}\in u
\\
\;\;\lra\;\;
\bigl\{x:\{y:x<y\}\in v\bigr\}\in u
\;\vee\:
\bigl\{x:\{y:x=y\}\in v\bigr\}\in u
\\
\;\;\lra\;\;
u\,{\widetilde<}\,v
\;\vee\:
u\,{\widetilde=}\,v.
\qquad\qquad\qquad\qquad\;\:
\end{align*}
And as easy to see, $u\,{\widetilde=}\,v$ 
means $u=v=\widetilde x$ for some~$x$.


\rule{0em}{0.4em}

The relations ${\widetilde>}$ and ${\widetilde\ge}$ are handled dually: 
by definition, $u\,{\widetilde>}\,v$ means $\{x:X_{<x}\in v\}\in u$; 
by~(\ref{table 2}), we get
\begin{align*}
\begin{array}{l|lll}
\{x:X_{<x}\in v\}\in u&
\;\supp(v)=\{y\}&\;\supp(v)=I_v&\;\supp(v)=J_v
\phantom{\underbrace{|}}
\\
\hline    
\quad\supp(u)=\{x\}&
\quad y<x&\sup I_v\le x&\inf J_v<x
\phantom{\overbrace{0}}    
\\
\quad\supp(u)=I_u&
\quad y<\sup I_u&\sup I_v<\sup I_u&\inf J_v<\sup I_u
\phantom{\overbrace{0}}    
\\
\quad\supp(u)=J_u&
\quad y\le\inf J_u&\sup I_v\le\inf J_u&\inf J_v\le\inf J_u
\phantom{\overbrace{0}}    
\end{array}
\end{align*}
and comparing this with~(\ref{ordering supports}), we see that 
$\{x:X_{<x}\in v\}\in u$ holds iff either 
$\supp(v)<\supp(u)$ or $\supp(u)=J_u=\supp(v)=J_v$, 
as required.
\rule{0em}{1em}
And $u\:{\widetilde\ge}\:v$ is equivalent to 
$u\:{\widetilde>}\:v\:\vee\:u\,{\widetilde=}\,v$.
\end{proof}


As easy to see from the established theorem, the relations extending 
linear orders generally have only a~few features of linear orders.

\begin{coro}
Let $X$~be a~linearly ordered set.
\\
1. 
For all non-principal ultrafilters~$u,v$ over~$X$, 
\begin{align*}
&(u\:\widetilde<\:v\,\vee\,u\:\widetilde>\:v)
\:\wedge\:
\neg\,(u\:\widetilde<\:v\,\wedge\,u\:\widetilde>\:v).
\end{align*}
More precisely, 
if\, $u,v$ have distinct supports, then 
\begin{align*}
u\:\widetilde<\:v
&\;\;\lra\;\;
v\:\widetilde>\:u
\;\;\lra\;\;
\neg\,(v\:\widetilde<\:u)
\;\;\lra\;\;
\neg\,(u\:\widetilde>\:v)
\;\;\lra\;\;
\supp(u)<\supp(v),
\end{align*}
and if\, $u,v$ have the same support, then
\begin{align*}
u\:\widetilde<\:v
&\;\;\lra\;\;
\neg\,(u\:\widetilde>\:v)
\;\;\lra\;\;
\supp(u)=\supp(v)=I_u=I_v,
\\
u\:\widetilde>\:v
&\;\;\lra\;\;
\neg\,(u\:\widetilde<\:v)
\;\;\lra\;\;
\supp(u)=\supp(v)=J_u=J_v.
\end{align*}
2. 
The relations $\widetilde<$, $\widetilde\le$, 
$\widetilde>$, $\widetilde\ge$ are transitive, 
but non-antisymmetric, non-connected, and 
neither reflexive nor irreflexive.
\end{coro}

\begin{proof} 
1. 
It immediately follows from Theorem~1. 
Alternatively, we can see this without Theorem~1, 
from a~general argument: start from the corresponding 
formula about $<$ and~$>$ and observe that connectives 
commute with ultrafilter quantifiers.

2. 
Transitivity is also immediate by Theorem~1. 
Moreover, by clause~1, we have the following 
description of points of reflexivity and irreflexivity:
\begin{align*}
u\:\widetilde<\:u
&\;\;\lra\;\;
\neg\,(u\:\widetilde>\:u)
\;\;\lra\;\;
\supp(u)=I_u,
\\
u\:\widetilde>\:u
&\;\;\lra\;\;
\neg\,(u\:\widetilde<\:u)
\;\;\lra\;\;
\supp(u)=J_u,
\end{align*} 
and if we pick $u\ne v$, the following equivalences describe
non-antisymmetry:
\begin{align*}
u\:\widetilde<\:v\,\wedge\,v\:\widetilde<\:u
&\;\;\lra\;\;
\supp(u)=\supp(v)=I_u=I_v,
\\
u\:\widetilde>\:v\,\wedge\,v\:\widetilde>\:u
&\;\;\lra\;\;
\supp(u)=\supp(v)=J_u=J_v,
\end{align*}
and non-connectedness:
\begin{align*}
\neg\,(u\:\widetilde<\:v)\,\wedge\,\neg\,(v\:\widetilde<\:u)
&\;\;\lra\;\;
\supp(u)=\supp(v)=J_u=J_v,
\\
\neg\,(u\:\widetilde>\:v)\,\wedge\,\neg\,(v\:\widetilde>\:u)
&\;\;\lra\;\;
\supp(u)=\supp(v)=I_u=I_v.
\end{align*}
Of course, the existence of two distinct ultrafilters $u,v$ 
with any of the required properties assumes a~dose of~$\AC$, 
as even the existence of one such ultrafilter does. 
In some cases (e.g.~if $X$~is well-orderable), the existence 
of one such ultrafilter implies the existence of two ultrafilters.
\end{proof}


Let us emphasize that, although for $u\ne v$ 
the formula $u\:\widetilde<\:v\,\vee\,u\:\widetilde>\:v$ 
looks like connectedness and the formula
$\neg\,(u\:\widetilde<\:v\,\wedge\,u\:\widetilde>\:v)$ 
looks like antisymmetry, 
they actually are not these properties since 
$u\,\widetilde\le\,v$ and $v\,\widetilde\ge\,u$ 
are not the same.
Instructively, this shows that the ultrafilter extension 
of a~relation does not commute with taking of the inverse.

Combining $u\,\widetilde\le\,v$ and $v\,\widetilde\ge\,u$, 
however, we can get a~kind of their ``commutator", which 
behaves closer to a~linear order. 
Define a~relation~$\trianglelefteq$ on ultrafilters by
\begin{align*}
u\trianglelefteq v
\;\;\lra\;\;
u\:\widetilde\le\:v\;\vee\;v\:\widetilde\ge\:u.
\end{align*}
It is clear from the previous that $u\trianglelefteq v$ 
is equivalent to $u\:\widetilde<\:v\:\vee\:
v\:\widetilde>\:u\:\vee\:\exists x\,(u=v=\tilde x)$.
We put also 
\begin{align*}
u\equiv v
\;\;\lra\;\;
u\trianglelefteq v\;\wedge\;v\trianglelefteq u.
\end{align*}

\begin{coro}
For all ultrafilters~$u,v$ over a~linearly ordered set~$X$,
\begin{align*}
u\trianglelefteq v
\;\;\lra\;\;
\supp(u)\le\supp(v),
\\
u\equiv v
\;\;\lra\;\;
\supp(u)=\supp(v).
\end{align*}
Thus $\trianglelefteq$ is a~linear pre-order,\, $\equiv$~is an equivalence, 
and the quotient set $\scc X/\!\equiv$ with the induced linear order is 
isomorphic to the set~$s(X)$ of supports with their natural ordering.
\end{coro}

\begin{proof} 
By Theorem~1, we have
\begin{align*}
u\trianglelefteq v
&\;\;\lra\;\;
u\:\widetilde\le\:v\;\vee\;v\:\widetilde\ge\:u
\\
&\;\;\lra\;\;
\supp(u)<\supp(v)
\;\vee\;\supp(u)=\supp(v)=I_u=I_v
\\
&\phantom{\;\;\lra\;\;
\supp(u)<\supp(v)}
\;\vee\;\supp(u)=\supp(v)=J_u=J_v
\;\vee\;\exists x\,(u=v=\tilde x)
\\
&\;\;\lra\;\;
\supp(u)\le\supp(v),
\end{align*}
as required.
The equivalence class $\{v:v\equiv u\}$ of~$u$ is 
hence $\{v:\supp(v)=\supp(u)\}$, and the claim follows.
\end{proof}

\begin{coro}
If $\le$~is a~well-order, then $\widetilde\le$~coincides
with~$\trianglelefteq$ and is a~pre-well-order.
\end{coro}

\begin{proof}
As noted above, for all non-principal ultrafilters~$u$ over a~well-ordered set~$X$,
$\supp(u)=I_u$. Hence, $u\:\widetilde\le\:v$ is equivalent to $\supp(u)\le\supp(v)$ 
by Theorem~1 and thus to $u\trianglelefteq v$ by Corollary~2.
\end{proof}


\subsection*{4. Ultrafilter extensions of operations min and max}

Here we describe the ultrafilter extensions of the minimum and maximum operations 
on a~given linearly ordered set. Firstly we do this in terms of the extensions of 
the order and the converse order.

\begin{thm}
If $X$~is a~linearly ordered set and $u,v$~are ultrafilters over~$X$, then
\begin{align*}
\widetilde\min(u,v)=u
\;\;\lra\;\;
&\widetilde\max(u,v)=v
\;\;\lra\;\;
u\:{\widetilde\le}\:v\:\vee\:u=v,
\\
\rule{0em}{1.2em}
\widetilde\min(u,v)=v
\;\;\lra\;\;
&\widetilde\max(u,v)=u
\;\;\lra\;\;
u\:{\widetilde\ge}\:v\:\vee\:u=v.
\end{align*}
\end{thm}

\begin{proof} 
We have, for all $S\subseteq X$,
\begin{align*}
S\in\widetilde\min(u,v)
\;\;\lra\;\;
\bigl\{x:\{y:\min(x,y)\in S\}\in v\bigr\}\in u
\qquad\qquad\qquad\qquad\;\;\;
\\
\;\;\lra\;\;\,
\bigl\{x:\{y:
(x\le y\,\wedge\,x\in S)
\,\vee\,
(x\ge y\,\wedge\,y\in S)
\}\in v\bigr\}\in u
\qquad\qquad\;\:
\\
\;\;\lra\;\;
\bigl(\bigl\{x:\{y:x\le y\}\in v\bigr\}\in u
\,\wedge\,S\in u\bigr)
\:\vee\:
\bigl(\bigl\{x:\{y:x\ge y\}\in v\bigr\}\in u
\,\wedge\,S\in v\bigr)\,\:
\\
\;\;\lra\;\;
\bigl(u\:{\widetilde\le}\,v\,\wedge\,S\in u\bigr)
\:\vee\,
\bigl(u\:{\widetilde\ge}\,v\,\wedge\,S\in v\bigr).
\qquad\qquad\qquad\qquad\;\;\;
\end{align*}
Therefore, 
\begin{align*}
\widetilde\min(u,v)=
\left\{
\begin{array}{cl}
u&\text{if }\;
u\:{\widetilde\le}\:v,
\\
\rule{0em}{1.2em}
v&\text{if }\;
u\:{\widetilde\ge}\:v.
\end{array}
\right.
\end{align*}
The dual argument gives
\begin{align*}
\widetilde\max(u,v)=
\left\{
\begin{array}{cl}
u&\text{if }\;
u\:{\widetilde\ge}\:v,
\\
\rule{0em}{1.2em}
v&\text{if }\;
u\:{\widetilde\le}\:v.
\end{array}
\right.
\end{align*}
Recalling now that for any~$u\in\scc X$ either 
$u\:{\widetilde\le}\:u$ or $u\:{\widetilde\ge}\:u$
(if $u$~is non-principal, this depends on what of $I_u$ or~$J_u$ 
is the support of~$u$), we complete the proof.
\end{proof}

Now we are able to describe $\widetilde\min$ and $\widetilde\max$ 
in terms of supports.

\begin{coro}
If $X$~is a~linearly ordered set and $u,v$~are ultrafilters over~$X$, 
then
\begin{align*}
\widetilde\min(u,v)=u
&\;\;\lra\;\;
\widetilde\max(u,v)=v
\\
&\;\;\lra\;\;
\supp(u)<\supp(v)\:\vee\:\supp(u)=\supp(v)=I_u=I_v\:\vee\:u=v,
\\
\rule{0em}{1.2em}
\widetilde\min(u,v)=v
&\;\;\lra\;\;
\widetilde\max(u,v)=u
\\
&\;\;\lra\;\;
\supp(v)<\supp(u)\:\vee\:\supp(u)=\supp(v)=J_u=J_v\:\vee\:u=v.
\end{align*}
\end{coro}

\begin{proof}
Theorems~1 and~2. 
\end{proof}

\begin{ex}
If $X$ is $\omega$ with the natural ordering, we get
$\widetilde\max(u,v)=v$ if $v$~is non-principal,
and
$\widetilde\max(u,v)=u$ if $v$~is principal and $u$~non-principal.
This was noted in~\cite{Hindman Strauss}, Exercise~4.1.11.
\end{ex}


Turning to algebraic properties of $\widetilde\min$ 
and~$\widetilde\max$, we recall some facts about skew algebras.
$(X,\,\cdot\,)$ is a~{\it skew semilattice\/}, or shorter, 
a~{\it band\/}, iff $\,\cdot\,$~is associative and idempotent, 
and a~{\it semilattice\/} iff it is moreover commutative. 
A~band is {\it rectangular\/}, or {\it nowhere commutative\/}, 
iff it satisfies $xyx=x$, or equivalently, $xy\ne yx\vee x=y$.
Bands satisfying the stronger condition $xy=x\,\vee\,xy=y$ are 
sometimes called {\it quasi-trivial\/}, see e.g.~\cite{Kepka};
they are easily characterized as groupoids (i.e.~algebras with one 
binary operation) in which each non-empty subset forms a~subgroupoid.
A~complete description of all varieties of bands can be found in 
any of \cite{Biryukov}--\cite{Gerhard}; for more on various special 
classes of semigroups see e.g.~\cite{Nagy}. The congruence~$D$ 
on a~band~$X$ is defined by letting, for all $x,y\in X$, 
$$x\,D\,y\,\;\;\lra\;\;\,xyx=x\:\wedge\:yxy=y.$$ 
The quotient $X/D$ of a~band~$X$ is a~semilattice and $D$-equivalence 
classes are rectangular subbands of~$X$; moreover, $X/D$~is the largest 
semilattice quotient of~$X$ (i.e.~any homomorphism of~$X$ into any 
semilattice~$Y$ is decomposed into the canonical homomorphism of $X$ 
onto~$X/D$ and a~homomorphism of~$X/D$ into~$Y$) and the $D$-equivalence 
class of each $x\in X$ is the largest rectangular subband containing~$x$.

$(X,+,\,\cdot\,)$ is a~{\it skew lattice\/} iff both $(X,+)$ and 
$(X,\,\cdot\,)$ are bands and the following absorption laws hold:
\begin{align*}
x(x+y)&=x+xy=x,
\\
(x+y)y&=xy+y=y.
\end{align*}
A~commutative skew lattice is a~{\it lattice\/}. A~skew lattice is
{\it rectangular\/} iff both its bands are rectangular and dualize 
each other: $x+y=yx$. In a~skew lattice~$X$, the congruences~$D$ for $+$ 
and~$\,\cdot\,$ coincide, $X/D$~is the largest lattice quotient of~$X$, 
and $D$-equivalence classes are maximal rectangular skew lattices.
A~skew lattice is {\it distributive\/} iff each of its operations 
is left and right distributive w.r.t.~another one:
\begin{align*}
x(y+z)&=xy+xz,
\qquad
x+yz=(x+y)(x+z),
\\
(x+y)z&=xz+yz,
\qquad
xy+z=(x+z)(y+z).
\end{align*}
(Note that distributivity implies ``a~half" of the absorption identitiess 
above.) If a~skew lattice~$X$ is distributive, so is the lattice~$X/D$.
We point out that skew lattices, introduced (with slightly different 
absorption laws) in~\cite{Jordan}, were intensively studied in past 
decades, see e.g.~\cite{Leech 89,Leech 96}.


Our following result shows that the ultrafilter extensions of 
linearly ordered sets with its minimum and maximum operations 
provide natural instances of skew lattices.


\begin{coro}
The operations $\widetilde\min$ and\, $\widetilde\max$ are associative, 
idempotent and even quasi-trivial, non-commutative, and distributive 
w.r.t.~each other. Therefore, $(\scc X,\widetilde\min,\widetilde\max)$ is 
a~distributive skew lattice.
\end{coro}

\begin{proof}
As well-known, associativity is stable under ultrafilter
extensions (see~\cite{Hindman Strauss}), so the operations 
$\widetilde\min$ and $\widetilde\max$ are associative. 
On the other hand, it can be shown that neither commutativity, 
nor idempotency, nor distributivity is not stable (see~\cite{Saveliev2}).

It is clear from Theorem~2 that $\widetilde\min$ and $\widetilde\max$
are indeed non-commutative on distinct non-principal ultrafilters with 
the same support. Indeed, if $\supp(u)=\supp(v)=I_u=I_v$, then 
$\widetilde\min(u,v)=\widetilde\max(v,u)=u$ and 
$\widetilde\min(v,u)=\widetilde\max(v,u)=v$, and similarly for the dual case.
So we get the following description of points of non-commutativity:
\begin{align}\label{non-commutativity}
\widetilde\min(u,v)\ne\widetilde\min(v,u)
\;\;\lra\;\;
\widetilde\max(u,v)\ne\widetilde\max(v,u)
\;\;\lra\;\;
u\ne v\;\wedge\;\supp(u)=\supp(v).
\end{align}

It is evident from Theorem~2 also that $\widetilde\min$ and $\widetilde\max$ 
are idempotent. But this can be seen without Theorem~2 from a~general fact: 
each of $\min$ and $\max$ satisfies quasi-triviality, which is obviously 
stronger than idempotency and is stable under the ultrafilter extension
(see~\cite{Saveliev2}).

Next, distributivity is verified by a~direct calculation. 
E.g.~check the identity
\begin{align}\label{distributivity}
\widetilde\max\bigl(u,\widetilde\min(v,w)\bigr)=
\widetilde\min\bigl((\widetilde\max(u,v),\widetilde\max(u,w)\bigr).
\end{align}
Recall that we have either $u\:\widetilde\le\:v$ 
or $u\:\widetilde\ge\:v$, but if one of $u,v$ is non-principal, 
not both (Corollary~1). 
Hence, for any $u,v,w$ we have exactly 8~conjunctions 
of possible relationships between each pair of them:
\begin{align*}
\text{(i)}\;\;
u\:\widetilde\le\:v
\;\wedge\;
u\:\widetilde\le\:w
\;\wedge\;
v\:\widetilde\le\:w,
&\qquad
\,\,\,\text{(v)}\;\;
u\:\widetilde\ge\:v
\;\wedge\;
u\:\widetilde\le\:w
\;\wedge\;
v\:\widetilde\le\:w,
\\
\text{(ii)}\;\;
u\:\widetilde\le\:v
\;\wedge\;
u\:\widetilde\le\:w
\;\wedge\;
v\:\widetilde\ge\:w,
&\qquad
\,\,\text{(vi)}\;\;
u\:\widetilde\ge\:v
\;\wedge\;
u\:\widetilde\le\:w
\;\wedge\;
v\:\widetilde\ge\:w,
\\
\text{(iii)}\;\;
u\:\widetilde\le\:v
\;\wedge\;
u\:\widetilde\ge\:w
\;\wedge\;
v\:\widetilde\le\:w,
&\qquad
\,\text{(vii)}\;\;
u\:\widetilde\ge\:v
\;\wedge\;
u\:\widetilde\ge\:w
\;\wedge\;
v\:\widetilde\le\:w,
\\
\text{(iv)}\;\;
u\:\widetilde\le\:v
\;\wedge\;
u\:\widetilde\ge\:w
\;\wedge\;
v\:\widetilde\ge\:w,
&\qquad
\text{(viii)}\;\;
u\:\widetilde\ge\:v
\;\wedge\;
u\:\widetilde\ge\:w
\;\wedge\;
v\:\widetilde\ge\:w.
\end{align*}
It is immediate from Theorem~2 that the identity~(\ref{distributivity}) 
holds in all of these cases except for cases (iii) and~(vi), where it 
may appear that it fails.
However, these two cases are in fact degenerate because 
of transitivity of $\widetilde\le$ and~$\widetilde\ge$ 
(Corollary~1). E.g.~in case~(iii),
$u\:\widetilde\le\:v\:\widetilde\le\:w$ 
gives $u\:\widetilde\le\:w$, which together with 
$u\:\widetilde\ge\:w$ gives $u=w=\tilde x$ for some~$x$, whence 
it follows $u=v=w$, which of course gives the required identity.

Finally, to handle absorption let check e.g.~that 
$$\widetilde\min\bigl(u,\widetilde\max(u,v)\bigr)=u.$$
But this easily follows from Theorem~2:
if $\widetilde\max(u,v)=u$ then the left term $\widetilde\min(u,u)$ equals~$u$
by idempotency, while if $\widetilde\max(u,v)=v$ then the left term 
$\widetilde\min(u,v)$ equals~$u$ because $\widetilde\min(u,v)=u$ is equivalent 
to $\widetilde\max(u,v)=v$.
\end{proof}


Thus, among the obvious features of the operations $\min$ and~$\max$, 
only commutativity fails under ultrafilter extensions. Modulo the 
equivalence~$\equiv$, however, the operations $\widetilde\min$ and 
$\widetilde\max$ become commutative and actually the corresponding 
minimum and maximum.

\begin{coro}
Let $X$~be a~linearly ordered set.
\\
1.
For all $u,v\in\scc X$,
\begin{align*}
u\trianglelefteq v
&\;\;\lra\;\;
\,\widetilde\min(u,v)=u\:\vee\:\widetilde\min(v,u)=u\,
\;\;\lra\;\;
\,\widetilde\min(u,v)=u\:\vee\:
\widetilde\min(u,v)\ne\widetilde\min(v,u)\,
\\
&\;\;\lra\;\;
\widetilde\max(u,v)=v\:\vee\:\widetilde\max(v,u)=v
\;\;\lra\;\;
\widetilde\max(u,v)=v\:\vee\:
\widetilde\max(u,v)\ne\widetilde\max(v,u),
\end{align*}
and
\begin{align*}
u\equiv v
\;\;\lra\;\;
\widetilde\min(u,v)\ne\widetilde\min(v,u)\:\vee\:u=v
\;\;\lra\;\;
\widetilde\max(u,v)\ne\widetilde\max(v,u)\:\vee\:u=v.
\end{align*}
2. 
The equivalence~$\equiv$ is a~congruence of the skew lattice
$(\scc X,\widetilde\min,\widetilde\max)$ and actually coincides 
with~$D$. 
\\
3. 
The quotient $\scc X/\!\equiv$ with the operations induced by 
$\widetilde\min$ and $\widetilde\max$ is isomorphic to the lattice $s(X)$ 
with its minimum and maximum operations.
\\
4. 
The $\equiv$-equivalence class of each $u\in\scc X$ is either 
a~left-zero band for~$\widetilde\min$ and a~right-zero band for~$\widetilde\max$,
or conversely,
a~right-zero band for~$\widetilde\min$ and a~left-zero band for~$\widetilde\max$.
\end{coro}

\begin{proof}
1.
By reflexivity of~$\trianglelefteq$ and Theorem~2, we have
\begin{align*}
u\trianglelefteq v
&\;\;\lra\;\;
u\:\widetilde\le\:v\,\vee\,v\:\widetilde\ge\:u
\\
&\;\;\lra\;\;
u\:\widetilde\le\:v\,\vee\,v\:\widetilde\ge\:u\,\vee\,u=v
\;\;\lra\;\;
\widetilde\min(u,v)=u\,\vee\,\widetilde\min(v,u)=u.
\end{align*}
The equivalences
$$
u\equiv v\;\;\lra\;\;
\widetilde\min(u,v)\ne\widetilde\min(v,u)\,\vee\,u=v
$$
and 
$$
u\trianglelefteq v\;\;\lra\;\;
\widetilde\min(u,v)\ne\widetilde\min(v,u)\,\vee\,\widetilde\min(u,v)=u
$$
can be deduced either directly from Theorem~2 or by using~(\ref{non-commutativity}).
The characterizations using~$\widetilde\max$ are obtained similarly.

2,~3.
By Corollary~3, $u\equiv v$ is equivalent to $\supp(u)=\supp(v)$, 
which holds for non-principal~$u,v$ either if the support is $I_u=I_v$ 
or if it is $J_u=J_v$. Now it is easily follows from Corollary~4 that 
$\equiv$~is a~congruence of $(\scc X,\widetilde\min,\widetilde\max)$ 
and its quotient is isomorphic the lattice $(s(X),\min,\max)$. 
As the congruence~$D$ has the largest lattice quotient, we conclude that 
$D\subseteq{\equiv}$. To verify the converse inclusion ${\equiv}\subseteq D$,
it suffices to show the following implication: 
$$
u\equiv v\;\,\to\,\;
\widetilde\min\bigl(\widetilde\min(u,v),u\bigr)=u.
$$
This is immediate in the case $\widetilde\min(u,v)=u$ as well as 
in the case $\widetilde\min(u,v)=v\,\wedge\,\widetilde\min(v,u)=u$.
In the remaining case $\widetilde\min(u,v)=\widetilde\min(v,u)=v$ 
we use Theorem~2 to conclude that $u=v$ and so the implication holds too.

4.
By Corollary~4, if $\supp(u)=I_u$ then the $\equiv$-equivalence class 
of~$u$ is a~left-zero band for~$\widetilde\min$ and a~right-zero band 
for~$\widetilde\max$, and dually if $\supp(u)=J_u$. If $\supp(u)=\{x\}$
then of course the class is the singleton~$\{u\}$.
\end{proof}


We see that the ultrafilter extension of a~given linear order, as well as 
of other relations and operations definable via it, allows a~clear and easy 
description. This is so, roughly speaking, because the theory of linear orders 
is easy. Ultrafilters having the same supports behave in the same way, so 
each $\equiv$-equivalence class can be identified with the filter that is its 
intersection (i.e.~with a~filter that is generated either (i)~by one point, or 
(ii)~by all final segments of an initial segment without the last element, or 
else (iii)~by all initial segments of a~final segment without the first element).
These filters, in turn, can be identified with ultrafilters on the Boolean algebra 
of definable subsets, which has a~rather simple structure since the theory is easy.


\begin{task}
Study ultrafilter extensions of partially ordered sets and 
related algebras (semilattices, lattices, Boolean algebras, 
etc.), and also their skew generalizations.
\end{task}

It may be hypothesed that the extensions are again some skew algebras, 
however, a~proof requires new arguments since now ultrafilters can be 
concentrated on antichains.


\begin{ack}
I~am grateful to Nikolay Poliakov for his query about ultrafilter extensions 
of linear orders and the operations of minimum and maximum, stimulated me to 
write this note, and for his comments on my manuscript. I~thank Ganna Kudryavtseva 
for her question about relationships of $\equiv$ to~$D$, for informing me about 
recent studies in skew lattices, also for commenting the manuscript, and especially 
for her hospitality during my visit to the Ljubljana University, partially supported 
by ARRS grant P1-0288. Finally, I~acknowledge a~partial support of this research by 
RFBR grants 11-01-00958 and 11-01-93107, and NSh grant 5593-2012-1.                 
\end{ack}



\end{document}